\documentclass[11pt,a4paper]{amsart}

\usepackage{cite,url,hyperref}

\usepackage{color}

\definecolor{gr}{rgb}   {0.,   0.69,   0.23 }
\definecolor{bl}{rgb}   {0.,   0.5,   1. }
\definecolor{mg}{rgb}   {0.85,  0.,    0.85}
\definecolor{or}{rgb}   {0.9,  0.5,   0.}

\newcommand{\dd}{\mathrm{d}}
\renewcommand{\Tilde}{\widetilde}

\renewcommand{\Bar}{\overline}

\newcommand{\RR}{\mathbb{R}}

\newcommand{\NN}{\mathbb{N}}
\newcommand{\cO}{\mathcal{O}}
\newcommand{\cD}{\mathcal{D}}
\newcommand{\cH}{\mathcal{H}}

\newcommand{\mx}{\mathrm{max}}

\newcommand{\vol}{\mathop{\mathrm{Vol}}}
\newcommand{\area}{\mathop{\mathrm{Area}}}

\newcommand{\Bk}{\color{black}}

\newtheorem{theorem}{Theorem}
\newtheorem{prop}[theorem]{Proposition}
\newtheorem{cor}[theorem]{Corollary}
\newtheorem{lemma}[theorem]{Lemma}

\theoremstyle{definition}

\newtheorem*{oq}{Open question}

\begin{document}

\title{Mean curvature bounds and eigenvalues of~Robin Laplacians}

\author{Konstantin Pankrashkin}

\address{Laboratoire de math\'ematiques, Universit\'e Paris-Sud, B\^atiment 425, 91405 Orsay Cedex, France}

\email{konstantin.pankrashkin@math.u-psud.fr}
\urladdr{http://www.math.u-psud.fr/~pankrash/}

\author{Nicolas Popoff}

\address{Centre de Physique Th\'eorique, Campus de Luminy, Case 907, 13288 Marseille cedex 9, France}

\email{nicolas.popoff@cpt.univ-mrs.fr}
\urladdr{http://perso.univ-rennes1.fr/nicolas.popoff/}

\begin{abstract}
We consider the Laplacian with
 attractive Robin boundary conditions,
\[
Q^\Omega_\alpha u=-\Delta u, \quad \dfrac{\partial u}{\partial n}=\alpha u  \text{ on } \partial\Omega,
\]
in a class of bounded smooth domains $\Omega\in\mathbb{R}^\nu$; here $n$ is the outward unit normal and $\alpha>0$ is a constant.
We show that for each $j\in\mathbb{N}$ and $\alpha\to+\infty$, the $j$th eigenvalue $E_j(Q^\Omega_\alpha)$ has the asymptotics
\[
E_j(Q^\Omega_\alpha)=-\alpha^2 -(\nu-1)H_\mathrm{max}(\Omega)\,\alpha+{\mathcal O}(\alpha^{2/3}),
\]
where $H_\mathrm{max}(\Omega)$ is the maximum mean curvature at $\partial \Omega$.
The discussion of the reverse Faber-Krahn inequality
gives rise to a new geometric problem concerning the minimization of $H_\mathrm{max}$.
In particular, we show that the ball is the strict minimizer
of $H_\mathrm{max}$ among the smooth star-shaped domains of a given volume, which leads
to the following result:
if $B$ is a ball and $\Omega$ is any other star-shaped smooth domain of the same volume,
then for any fixed $j\in\mathbb{N}$
we have $E_j(Q^B_\alpha)>E_j(Q^\Omega_\alpha)$ for large $\alpha$.
An open question concerning a larger class of domains is formulated.
\end{abstract}

\maketitle

\section{Introduction}

Let $\Omega\subset \RR^\nu$, $\nu\ge 2$, be a bounded, connected, $C^3$ smooth domain; such domains will be called \emph{admissible}
through the text. For $\alpha\in\RR$, denote by $Q^\Omega_\alpha$ the self-adjoint operator in $L^2(\Omega)$
acting as $Q^\Omega_\alpha u =-\Delta u$ on the functions $u\in H^2(\Omega)$ satisfying the Robin boundary conditions
\[
\dfrac{\partial u} {\partial n} =\alpha u \text{ on } S:=\partial \Omega,
\]
where $n$ is the outward pointing unit normal vector at $S$.

More precisely, $Q^\Omega_\alpha$ is
the self-adjoint operator in $L^2(\Omega)$ associated with
the quadratic form $q^\Omega_\alpha$ defined on the domain
$\cD(q^\Omega_\alpha)=H^1(\Omega)$ by
\[
q^\Omega_\alpha(u,u)=\int_{\Omega} |\nabla u|^2\dd x -\alpha \int_S |u|^2\dd S,
\]
where $\dd S$ stands for the $(\nu-1)$-dimensional Hausdorff measure on $S$.
We remark that the boundary $S$ is not necessarily connected. 

Throughout the paper, for a lower semibounded 
 self-adjoint operator $Q$ with a compact resolvent
we denote by $E_j (Q)$ its $j$th eigenvalue when numbered in non-decreasing order and counted
according to the multiplicities.
The aim of the paper is to discuss some relations between the geometry
of $\Omega$ and the eigenvalues $E_j (Q^\Omega_\alpha)$ in
the asymptotic regime $\alpha\to +\infty$.

The first order asymptotics of the low lying eigenvalues has been investigated
in the previous papers  \cite{LaOckSa98,LouZhu04,LevPar08,DaKe10}
and it was shown that for any fixed $j\in\NN$ one has
\[
E_j(Q^\Omega_\alpha)=-\alpha^2 +o(\alpha^2), \quad \alpha\to+\infty.
\]
Some results concerning the convergence of
higher eigenvalues were obtained in~\cite{cch}.
We mention also several works dealing with non-smooth
boundaries~\cite{HP,LevPar08,McC}.
Only few works are dedicated to the study of subsequent terms in the asymptotic expansion.
In the two-dimensional case ($\nu=2$) it was shown \cite{ExMinPar14,Pank13}
that for any fixed $j\in\NN$ one has
\begin{equation}
      \label{eq-ejj}
E_j(Q^\Omega_\alpha)=-\alpha^2 -\gamma_{\mx}\alpha+\cO(\alpha^{2/3}), \quad \alpha\to+\infty,
\end{equation}
where $\gamma_\mx$ is the maximum curvature at the boundary. Some results for the principal eigenvalue of balls and spherical shells 
in arbitrary dimension were obtained in~\cite{FreKre14}. Our first aim is to find an analog of~\eqref{eq-ejj}
for general domains in higher dimensions, in particular we wonder what is the quantity which plays the role of $\gamma_{\mx}$
in dimension $\nu\geq 3$.

Let us introduce the necessary notation. As mentioned above, $s\mapsto n(s)$ is the Gauss map on $S$, i.e.
$n(s)$ is the outward pointing unit normal vector at $s\in S$.
Denote
\[
L_s:=\dd n(s): T_s S\to T_s S
\]
the shape operator at $s\in S$. Its eigenvalues
$\kappa_1,\dots,\kappa_{\nu-1}$ are the so-called principal curvatures at $s$,
and the mean curvature $H(s)$ at $s$ is defined by
\[
H(s)=\dfrac{\kappa_1(s)+\dots+\kappa_{\nu-1}}{\nu-1}\equiv \dfrac{1}{\nu-1}\, \mathop{\mathrm{tr}} L_s.
\]
Due to the assumption $S\in C^3$ we have at least $H\in C^1(S)$, and we will denote
\[
H_\mx\equiv H_\mx(\Omega):=\max_{s\in S} H(s).
\]

\begin{theorem}\label{thm1}
Let $\nu\geq2$ and $\Omega\subset\RR^\nu$ be an admissible domain,
then for any fixed $j\in \NN$ one has the asymptotics
\begin{equation}
       \label{eq-th1}
E_j(Q^\Omega_\alpha)=-\alpha^2 - (\nu-1) H_\mx(\Omega) \alpha +\cO(\alpha^{2/3}) \text{ as } \alpha\to+\infty.
\end{equation}
If, additionally, the boundary is $C^{4}$, then the remainder estimate $\cO(\alpha^{2/3})$ can be replaced by $\cO(\alpha^{1/2})$
\end{theorem}

The proof is given in Section~\ref{sec1}.

The result of Theorem~\ref{thm1} can be used to discuss some questions related to the so-called
reverse Faber-Krahn inequality. It was conjectured in \cite{Bar77} that for any $\alpha>0$
among the domains $\Omega$ of the same volume it is the ball which maximizes the first eigenvalue of $Q^\Omega_\alpha$.
In such a  setting, the conjecture appears to be false as the very recent counterexample of \cite{FreKre14} shows:
if $B$ is a ball of radius $r>0$ and $\Omega$ is a spherical shell of the same volume, then for large $\alpha$
one has $E_1(Q^B_\alpha)<E_1(Q^\Omega_\alpha)$. This result is easily visible with the help
of our result: one has $H_\mx(B)=1/r$ and $H_\mx(\Omega)=1/R$ with $R$ being the outer radius of $\Omega$. 
Therefore, $H_\mx(B)<H_\mx(\Omega)$, which gives the sought inequality.
However, some positive results are available for restricted classes of domains
and special ranges of~$\alpha$. In particular, the conjecture is true for domains which are in a sense close to a ball \cite{Bar77,FerNiTr14}
as well as for  small $\alpha$ and two-dimensional domains \cite{FreKre14}.
(We remark that the situation for negative $\alpha$ is completely different: the ball minimizes
the first eigenvalue for any $\alpha<0$, see \cite{Bos86,Dan06,BucAl14}).

With the help of Theorem~\ref{thm1} one can easily see that the study of Faber-Krahn-optimal domains 
for large $\alpha$ leads to the purely geometric question: to minimize
$H_\mx(\Omega)$ in a given class of domains. It seems that such a problem setting is new and may have its own interest. Some aspects of this problem will be discussed in~Section~\ref{sec3}.
In particular, we show the following lower bound, see Subsection~\ref{ssec6}:
\begin{theorem}
\label{thm2}
Let $\Omega\subset\RR^\nu$ be a bounded, star-shaped, $C^2$ smooth domain, then
\begin{equation}
\label{E:MHmx}
H_{\mx}(\Omega) \ge \left(\frac{\vol B_{\nu}}{\vol \Omega}\right)^{1/\nu}
\end{equation}
where $B_\nu$ is the unit ball in $\RR^\nu$ .
Moreover, the above lower bound is an equality if and only if $\Omega$ is a ball.
\end{theorem}
By combining the two theorems we arrive at the following spectral result, which can viewed
as an asymptotic version of the reverse Faber-Krahn inequality for star-shaped domains:
\begin{cor}\label{cor1}
Let $\Omega\subset\RR^\nu$ be admissible, star-shaped, different from a ball.
Let $B$ be a ball in $\RR^\nu$ with $\vol B=\vol \Omega$.
For any $j\in \NN$ there exists $\alpha_{0}>0$ such that for all $\alpha \ge \alpha_{0}$ there holds 
the strict inequality
\[
E_{j}(Q^{B}_\alpha)>
E_{j}(Q^\Omega_\alpha).
\]
\end{cor}
It would be interesting to understand whether Theorem~\ref{thm2} and hence Corollary~\ref{cor1}
can be extended to a larger class of domains:

\begin{oq} Let $\Omega\in\RR^\nu$
be a bounded, $C^2$ smooth domain with a~\emph{connected boundary} and let $B$ be a ball of the same volume.
Do we have $H_\mx(\Omega)\ge H_\mx(B)$?
\end{oq}

At last we remark that one of our motivations for the present work
came from various analogies between the Robin problem and the Neumann magnetic laplacian in two dimensions,
in particular, a certain analogue of \eqref{eq-ejj} holds for the magnetic case~\cite{BeSt,HeMo01}.
However, a magnetic field has a non-trivial interaction with the boundary in higher dimension \cite{HeMo04},
and the comparison seems to stop here as our result shows.

\section{Asymptotics of the eigenvalues}\label{sec1}

This section is dedicated to the proof of Theorem~\ref{thm1}.
The scheme of the proof is close to that of~\cite{ExMinPar14,Pank13}
and is based on the Dirichlet-Neumann bracketing in suitably chosen tubular neighborhoods of $S$ and on an explicit
spectral analysis of some one-dimensional operators.
We recall the max-min principle for the eigenvalues of self-adjoint operators, see e.g.~\cite[Sec.~12.1]{ksm}:
if $Q$ is a lower semibounded 
self-adjoint operator with a compact resolvent
in a Hilbert space $\cH$ and $q$ is its quadratic form, then for any $j\in\NN$
we have
\[
E_j(Q)=\max_{\psi_1, \dots, \psi_{j-1}\in\cH} \min_{\substack{u\in \cD(q),\, u\ne 0\\
u\perp \psi_1,\dots, \psi_{j-1}}} \dfrac{q(u,u)}{\langle u, u\rangle}.
\]

\subsection{Dirichlet-Neumann bracketing}
For $\delta>0$ denote
\[
\Omega_\delta:=\big\{
x\in \Omega: \, \inf_{s\in S} |x-s|<\delta
\},
\quad
\Theta_\delta:=\Omega\setminus \overline{\Omega_\delta}.
\]
Denote by $q^{\Omega,N,\delta}_\alpha$ and $q^{\Omega,D,\delta}_\alpha$
the quadratic forms defined by the same expression as $q^\Omega_\alpha$
but acting respectively on the domains
\begin{gather*}
\cD(q^{\Omega,N,\delta}_\alpha)=H^1(\Omega_\delta)\cup H^1(\Theta_\delta),
\quad
\cD(q^{\Omega,D,\delta}_\alpha)=\Tilde H^1_0(\Omega_\delta)\cup H^1_0(\Theta_\delta),\\
\widetilde H^1_0(\Omega_\delta):=\{f\in H^1(\Omega_\delta): \, f=0 \text{ at } \partial\Omega_\delta\setminus S \}
\end{gather*}
and denote by $Q^{\Omega,N,\delta}_\alpha$ and $Q^{\Omega,D,\delta}_\alpha$
the associated self-adjoint operators in $L^2(\Omega)$.
Due to the inclusions
\[
\cD(q^{\Omega,D,\delta}_\alpha)\subset\cD(q^{\Omega}_\alpha)\subset\cD(q^{\Omega,N,\delta}_\alpha)
\]
and according to the max-min principle, for each $j\in\NN$ we have the inequalities
\[
E_j(Q^{\Omega,N,\delta}_\alpha)\le E_j(Q^\Omega_\alpha)\le E_j(Q^{\Omega,D,\delta}_\alpha).
\]
Furthermore, we have the representations
\[
Q^{\Omega,N,\delta}_\alpha=B^{\Omega,N,\delta}_\alpha \oplus
(-\Delta)^N_{\Theta_\delta},
\quad
Q^{\Omega,D,\delta}_\alpha=B^{\Omega,D,\delta}_\alpha \oplus
(-\Delta)^D_{\Theta_\delta},
\]
where $B^{\Omega,N,\delta}_\alpha$ and $B^{\Omega,D,\delta}_\alpha$
are the self-adjoint operators in $L^2(\Omega_\delta)$
associated with the respective quadratic forms
\begin{gather*}
b^{\Omega,\star,\delta}_\alpha(u,u)=\int_{\Omega_\delta} |\nabla u|^2dx -\alpha \int_S |u|^2d S,
\quad \star\in\{N,D\},\\
 \cD(b^{\Omega,N,\delta}_\alpha)=H^1(\Omega_\delta),\quad
\cD(b^{\Omega,D,\delta}_\alpha)=\widetilde H^1_0(\Omega_\delta),
\end{gather*}
and  $(-\Delta)^N_{\Theta_\delta}$ and $(-\Delta)^D_{\Theta_\delta}$
denote respectively the Neumann and the Dirichlet Laplacian in $\Theta_\delta$.
As both Neumann and Dirichlet Laplacians are non-negative, we have the inequalities
\begin{equation}
       \label{eq-est1}
E_j(B^{\Omega,N,\delta}_\alpha)\le E_j(Q^\Omega_\alpha)\le E_j(B^{\Omega,D,\delta}_\alpha)
\text{ for all $j$ with } E_j(B^{\Omega,D,\delta}_\alpha)<0.
\end{equation}
We remark that these inequalities are valid for any value of $\delta$.

\subsection{Change of variables}
\label{S:CV}
In order to study the eigenvalues of the operators $B^{\Omega,N,\delta}_\alpha$
and $B^{\Omega,D,\delta}_\alpha$ we proceed first with a change variables in $\Omega_\delta$
with small $\delta$. The computations below are very similar to those
performed in~\cite{CarExKr04} for a different problem.

It is a well-known result of the differential geometry
that for $\delta$ sufficiently small
 the map $\Phi$ defined by
\[
\Sigma:=S\times (0,\delta)\ni (s,t)\mapsto \Phi(s,t)=s-tn(s)\in \Omega_\delta
\]
is a diffeomorphism.
The metric $G$ on  $\Sigma$ induced by this embedding is
\begin{equation}
      \label{eq-gg}
G=g\circ (I_s-tL_s)^2 + \dd t^2,
\end{equation}
where $I_s:T_s S\to T_s S$ is the identity map,
and $g$ is the metric on $S$ induced by the embedding in $\RR^\nu$.
The associated volume form $\dd\Sigma$ on $\Sigma$ is
\[
\dd\Sigma =|\det G|^{1/2}\dd s\, \dd t=\varphi(s,t)|\det g|^{1/2} \dd s \,\dd t=\varphi \, \dd S \,\dd t,
\]
where
\[
\dd S=|\det g|^{1/2}\dd s
\]
is the induced $(\nu-1)$-dimensional volume form on $S$, and the weight $\varphi$ is given by
\begin{equation}
      \label{eq-r1}
\varphi(s,t):=\big|\det (I_s-t L_s)\big|=
1- t \mathop{\mathrm{tr}}L_s +p(s,t)t^2\equiv
1-(\nu-1) H(s) t +p(s,t)t^2,
\end{equation}
with $p$ being a polynomial in $t$ with $C^1$ coefficients depending on $s$.

Consider the unitary map
\[
U: L^2(\Omega_\delta)\to L^2(\Sigma,\dd\Sigma),
\quad
Uf=f\circ \Phi,
\]
and the quadratic forms 
\[
c^{\star}_\alpha(f,f)=b^{\Omega,\star,\delta}_\alpha(U^{-1}f, U^{-1}f)
\quad \cD (c^\star_\alpha)=U \cD (b^{\Omega,\star,\delta}_\alpha),
\quad \star\in\{N,D\}.
\]
We have then
\begin{align*}
c^{N}_\alpha(u,u)&=\int_{\Sigma} G^{jk}\overline{\partial_j u}  \partial_k u\,\dd\Sigma_\alpha
-\alpha \int_S |u(s,0)|^2 \dd S, \quad \cD(c^{N}_\alpha)= H^1(\Sigma),\\
c^{D}_\alpha(u,u)&=\text{the restriction of $c^{N}_\alpha$ to }
\cD(c^{D}_\alpha)=\widetilde H^1_0(\Sigma),
\end{align*}
with
\[
\widetilde H^1_0(\Sigma):=\big\{f\in H^1(\Sigma): f(\cdot,\delta)=0\big\}, \quad
(G^{jk}):=G^{-1}.
\]
\subsection{Estimates of the metrics}
We remark that due to \eqref{eq-gg} we can estimate, with some $0<C_-<C_+$,
\[
C_- g^{-1} + \dd t^2\le G^{-1}\le C_+ g^{-1} + \dd t^2, \quad (g^{\rho\mu}):= g^{-1},
\]
which shows that we have the form inequalities
\begin{equation}
       \label{eq-cc}
c^-_\alpha\le c^{N}_\alpha
 \quad \text{ and} \quad
c^{D}_\alpha\le c^+_\alpha
\end{equation}
with
\[
\begin{aligned}
c^-_\alpha(u,u)&:=C_-\int_{\Sigma} g^{\rho \mu} \,\overline{\partial_\rho u} \,\partial_\mu u\,\dd\Sigma 
+\int_{\Sigma} |\partial_t u|^2\dd\Sigma-\alpha \int_S |u(s,0)|^2 \dd S,\\
& \quad \cD(c^-_\alpha)=\cD(c^{N}_\alpha)=H^1(\Sigma),\\
c^+_\alpha(u,u)&:=C_+\int_{\Sigma} g^{\rho \mu} \,\overline{\partial_\rho u} \,\partial_\mu u\,\dd\Sigma
+\int_{\Sigma} |\partial_t u|^2\dd\Sigma-\alpha \int_S |u(s,0)|^2 \dd S,\\
& \quad \cD(c^+_\alpha)=\cD(c^{D}_\alpha)=\widetilde H^1_0(\Sigma).
\end{aligned}
\]
In particular, if $C^-_\alpha$ and $C^+_\alpha$ are the self-adjoint operators
acting in $L^2(\Sigma,\dd\Sigma)$ and associated with the forms $c^-_\alpha$ and $c^+_\alpha$ respectively, then
it follows from \eqref{eq-est1} and \eqref{eq-cc} that
\begin{equation}
       \label{eq-est2}
E_j(C^-_\alpha)\le E_j(Q^\Omega_\alpha)\le E_j(C^+_\alpha)
\text{ for all $j$ with } E_j(C^+_\alpha)<0.
\end{equation}

In order to remove the weight $\varphi$ in $\dd\Sigma$ we introduce another unitary transform
\begin{gather*}
V:L^2\big(S\times(0,\delta),\dd S \,\dd t\big)\equiv L^2(\Sigma, \dd S \,\dd t)\to L^2(\Sigma,\dd\Sigma), \\
 (V f)(s,t)=\varphi(s,t)^{-1/2}f(s,t)
\end{gather*}
and the quadratic forms
\[
d^{\pm}_\alpha(f, f)=c^{\pm}_\alpha(Vf,Vf), \quad \cD(d^{\pm}_\alpha)=V^{-1}\cD(c^{\pm}_\alpha).
\]
 We remark that due to \eqref{eq-r1} we can choose $\delta$ sufficiently small to have
the representation
\begin{equation}
       \label{eq-r2}
\varphi(s,t)^{-1/2}=1+m(s)\, t + t^2P(s,t), \quad P\in C^1(\overline\Sigma), \quad  \varphi >0 \text{ in } \overline\Sigma,
\end{equation}
where we denote for the sake of brevity
\[
m(s):=\dfrac{\nu-1}{2}H(s).
\]
We can write
\[
\partial_\mu (V u)= \varphi^{-1/2} \partial_\mu u + u \,\partial_\mu \varphi^{-1/2},
\]
and
\begin{multline}
     \label{eq-ee1}
\int_{\Sigma} g^{\rho \mu} \overline{\partial_\rho (V u)} \partial_\mu (V u)\,\dd\Sigma=
\int_{\Sigma} g^{\rho \mu} \overline{\partial_\rho (V u)} \partial_\mu (V u) \,\varphi\, \dd S \dd t\\
=\int_{\Sigma} \Big[\varphi^{-1}g^{\rho \mu} \overline{\partial_\rho u} \partial_\mu u
+g^{\rho\mu} \big(\overline u \partial_\rho \varphi^{-1/2}\big)  \varphi^{-1/2}\partial_\mu u\\
+g^{\rho\mu} \varphi^{-1/2}\overline{\partial_\rho u} \big(u \partial_\mu \varphi^{-1/2}\big)  
+|u|^2 g^{\rho \mu }  \overline{\partial_\rho \varphi^{-1/2}}{\partial_\mu \varphi^{-1/2}} \Big] \varphi\,\dd S \dd t.
\end{multline}

As $(g^{\rho\mu})$ is positive definite, using the Cauchy-Schwarz inequality we estimate
\begin{multline*}
\bigg|\int_{\Sigma}g^{\rho\mu} \,\big(\overline u \partial_\rho \varphi^{-1/2}\big)  \big(\varphi^{-1/2}\partial_\mu u \big)\,\varphi\,\dd S \dd t\bigg|^2\\
\le
\int_{\Sigma} \big( g^{\rho \mu} \,\overline{\varphi^{-1/2}\partial_\rho u} \,\varphi^{-1/2}\partial_\mu u\big)\,\varphi\,\dd S \,\dd t \cdot
\int_{\Sigma} |u|^2\Big( g^{\rho \mu} \,\overline{\partial_\rho \varphi^{-1/2}} \,\partial_\mu \varphi^{-1/2}\Big)\varphi\,\dd S \,\dd t\\
=
\int_{\Sigma} g^{\rho \mu} \,\overline{\partial_\rho u} \,\partial_\mu u\,\dd S \,\dd t\cdot
\int_{\Sigma} W|u|^2\,\dd S \,\dd t,
\quad
W:= \varphi\,g^{\rho \mu} \,\overline{\partial_\rho \varphi^{-1/2}} \,\partial_\mu \varphi^{-1/2}.
\end{multline*}
Substituting this inequality into \eqref{eq-ee1} we obtain
\begin{multline*}
\int_{\Sigma} g^{\rho \mu} \,\overline{\partial_\rho u} \,\partial_\mu u \,\dd S \,\dd t
- \int_{\Sigma} W|u|^2\dd S \,\dd t
\le
\int_{\Sigma} g^{\rho \mu} \,\overline{\partial_\rho (V u)} \,\partial_\mu (V u)\,\dd\Sigma\\
\le
2\int_{\Sigma} g^{\rho \mu} \,\overline{\partial_\rho u} \,\partial_\mu u \,\dd S \,\dd t
+2 \int_{\Sigma} W|u|^2\dd S \,\dd t.
\end{multline*}
Denoting 
\[
C_0:=\|W\|_{L^\infty(\Sigma)}
\]
we arrive finally at
\begin{multline*}
\int_{\Sigma} g^{\rho \mu} \,\overline{\partial_\rho u} \,\partial_\mu u\,\dd S \,\dd t
- C_0\int_{\Sigma} u|^2\dd S \dd t
\le
\int_{\Sigma} g^{\rho \mu} \overline{\partial_\rho (V u)} \partial_\mu (V u)\,\dd\Sigma\\
\le
2\int_{\Sigma} g^{\rho \mu} \,\overline{\partial_\rho u} \,\partial_\mu u \,\dd S \,\dd t
+2 C_0\int_{\Sigma} |u|^2\dd S \, \dd t.
\end{multline*}
On the other hand, as follows from \eqref{eq-r2},
\[
\partial_t \varphi^{-1/2}=m +t R \text{ with } R\in C^1(\Bar \Sigma),
\]
and  we have
\begin{multline*}
\int_{\Sigma} \big|\partial_t (Vu)\big|^2\dd\Sigma=\int_\Sigma \varphi \Big| \varphi^{-1/2}\partial_tu  +(m+t R)u\Big|^2\dd S \dd t\\
=\int_\Sigma |\partial_t u|^2\dd S \dd t
+\int_\Sigma (m+tR) \varphi^{1/2}\partial_t |u|^2\dd S \dd t
+\int_\Sigma  (m+tR)^2 \varphi |u|^2\dd S \dd t.
\end{multline*}
By setting
\[
B_1:=\big\|(m+tR)^2\varphi\big\|_{L^\infty(\Sigma)}
\]
we have
\[
-B_1 \int_\Sigma |u|^2\dd S \,\dd t\le
\int_\Sigma  (m+tR)^2 \varphi |u|^2\dd S \,\dd t
\le
B_1\int_\Sigma |u|^2\dd S \,\dd t.
\]
Furthermore, using the integration by parts we arrive at
\begin{multline*}
\int_\Sigma (m+tR)\varphi^{1/2} \partial_t |u|^2\dd S \,\dd t=\int_S  \Big(\int_0^\delta
(m+tR) \varphi^{1/2}\partial_t |u|^2\dd t \Big)\dd S\\
=\int_S \varphi(s,\delta)^{1/2} \big(m(s)+\delta R(s,\delta)\big) \big|u(s,\delta)\big|^2\dd S
-\int_S m(s) \big|u(s,0)\big|^2\dd S\\
-\int_\Sigma \partial_t \big((m+tR) \varphi^{1/2}\big) |u|^2 \dd S \,\dd t.
\end{multline*}
Set
\[
\beta:=\big\|\varphi(\cdot,\delta)^{1/2} \big(m(\cdot)+\delta R(\cdot,\delta)\big)\big\|_{L^\infty(S)},
\quad
B_2:= 
\Big\|\partial_t \big((m+tR) \varphi^{1/2}\big)\Big\|_{L^\infty(\Sigma)}.
\]
By summing up all the terms we see that for any $u\in H^1\big(S\times (0,\delta)\big)$
we have the minoration
\begin{multline*}
d^-_\alpha(u,u)\ge A_-\int_{S\times(0,\delta)} g^{\rho\mu} \,\overline{\partial_\rho u}\,\partial_\mu u 
\,\dd S \,\dd t
+\int_{S\times(0,\delta)} |\partial_t u|^2 \dd S \,\dd t\\
+B_- \int_{S\times(0,\delta)} | u|^2 \dd S \,\dd t
 -\int_S \big(\alpha+m(s)\big) |u(s,0)|^2 \dd S- \beta \int_S |u(s,\delta)|^2 \dd S
\end{multline*}
with
\[
A_-:=C_-, \quad B_-:=-C_- C_0-B_1-B_2.
\]
Using the quantity
\[
m_\mx:=\max_{s\in S} m(s)\equiv \dfrac{\nu-1}{2} H_\mx,
\]
and the inequality
\[
-\big(\alpha+m(s)\big)\ge -(\alpha+m_\mx), \quad s\in S,
\]
we estimate finally, in the sense of forms,
\[
k^-_\alpha \le d^-_\alpha
\]
with
\begin{align*}
k^-_\alpha(u,u):=&\, A_-\int_{S\times(0,\delta)} g^{\rho\mu} \overline{\partial_\rho u}\partial_\mu u \,\dd S \,\dd t
+\int_{S\times(0,\delta)} |\partial_t u|^2 \dd S \,\dd t\\
&\,+B_- \int_{S\times(0,\delta)} | u|^2 \dd S \,\dd t
 -(\alpha+m_\mx)\int_S  |u(s,0)|^2 \dd S\\
 &\,- \beta \int_S |u(s,\delta)|^2 \dd S,\\
\cD(k^-_\alpha)=&H^1\big(S\times (0,\delta)\big).
\end{align*}
We also have the majoration
\[
d^+_\alpha\le k^+_\alpha
\]
with 
\begin{align*}
k^+_\alpha(u,u)=&\,A_+\int_{S\times(0,\delta)} g^{\rho\mu} \,\overline{\partial_\rho u}\,\partial_\mu u \,\dd S \,\dd t
+\int_{S\times(0,\delta)} |\partial_t u|^2 \dd S \,\dd t\\
&\,+B_+ \int_{S\times(0,\delta)} | u|^2 \dd S \,\dd t
 -\int_S \big(\alpha+m(s)\big) |u(s,0)|^2 \dd S,\\
\cD(k^+_\alpha)=&\,\widetilde H^1_0\big(S\times (0,\delta)\big):=\big\{f\in H^1\big(S\times (0,\delta)\big):\, f(\cdot,\delta)=0\big\},
\end{align*}
where
\[
A_+:=2C_+ >0, \quad B_+:=2C_+B_0+B_1+B_2.
\]
Denote by $K^\pm_\alpha$ the self-adjoint operators in $L^2\big(S\times(0,\delta),\dd S \dd t\big)$
associated with the respective forms $k^\pm_\alpha$ then,
due to the preceding constructions, for any $j\in\NN$ we have the inequalities
\[
E_j(K^-_\alpha)\le E_j(C^-_\alpha), \quad E_j(C^+_\alpha)\le E_j(K^+_\alpha),
\]
which implies, due to \eqref{eq-est2},
\begin{equation}
      \label{eq-est3}
E_j(K^-_\alpha)\le E_j(Q^\Omega_\alpha) \le E_j(K^+_\alpha)
\text{ for all $j$ with } E_j(K^+_\alpha)<0.
\end{equation}

We remark that in what precedes the parameter $\delta$ was chosen
sufficiently small but fixed.

\subsection{Analysis of the reduced operators and lower bound} Let us analyze the eigenvalues of $K^-_\alpha$.
We remark that it can be represented as
\[
K^-_\alpha=(A_-L+ B_-)\otimes 1 + 1 \otimes T^-_\alpha,
\]
where $L$ is the (positive) Laplace-Beltrami operator on $S$ and acting in $L^2(S, \dd S)$
and $T^-_\alpha$ is the self-adjoint operator in $L^2(0,\delta)$
associated with the quadratic form
\[
\tau^-_\alpha(f,f)=\int_0^\delta \big|f'(t)\big|^2\dd t- (\alpha+m_\mx) \big|f(0)\big|^2
-\beta \big|f(\delta)\big|^2, \quad
\cD(\tau^-_\alpha)=H^1(0,\delta).  
\]
Clearly, the eigenvalues of $A_-L+ B_-$ do not depend on $\alpha$.
On the other hand, the eigenvalues of $T^-_\alpha$ in the limit
$\alpha\to+\infty$ were already analyzed in the previous works,
and one has the following result, see e.g. \cite[Lemma~3]{Pank13}
or \cite[Lemma~2.5]{ExMinPar14}:
\[
E_1(T^-_\alpha)=-(\alpha+m_\mx)^2 + o(1) \text{ and }
E_2(T^-_\alpha)\ge 0 \text{ as } \alpha\to+\infty.
\]
It follows that for any fixed $j\in\NN$ and  $\alpha\to+\infty$ we have
\begin{multline}
     \label{eq-eee}
E_j(K^-_\alpha)=E_1(T^-_\alpha)+  A_{-}E_j(L)+B_{-} \Bk =-(\alpha+m_\mx)^2+\cO(1)\\
=-\alpha^2-2m_\mx \alpha+\cO(1)
=-\alpha^2 - (\nu-1) H_\mx \,\alpha+ \cO(1).
\end{multline}
Combining this with \eqref{eq-est3} provides the lower bound of \eqref{eq-th1}.

\subsection{Construction of quasi-modes and upper bound} It remains to analyze the eigenvalues of $K^+_\alpha$.
Let us show first an additional auxiliary estimate.
\begin{lemma}\label{lem2}
Let $T^+_\alpha$ denote the self-adjoint operator in $L^2(0,\delta)$
associated with the quadratic form
\begin{align*}
\tau^+_\alpha(f,f)&:=\int_0^\delta \big|f'(t)\big|^2\dd t- (\alpha+m_\mx) \big|f(0)\big|^2, \\
\cD(\tau^+_\alpha)&=\big\{f\in H^1(0,\delta): \, f(\delta)=0\big\}. 
\end{align*}
Then
\begin{equation}
    \label{eq-tpl}
E_1(T^+_\alpha)=-(\alpha+m_\mx)^2+o(1) \text{ and } E_2(T^+_\alpha)\ge 0 \text{ as } \alpha\to +\infty.
\end{equation}
Furthermore, if $\psi_\alpha$ is a normalized eigenfunction for the eigenvalue $E_1(T^+_\alpha)$,
then
\begin{equation}
   \label{eq-pp1}
\big|\psi_\alpha(0)\big|^2=\cO(\alpha) \text{ as } \alpha\to+\infty.
\end{equation}
\end{lemma}

\begin{proof}
The asymptotics \eqref{eq-tpl} were already obtained in the previous works, see e.g. \cite[Lemma~4]{Pank13}
or \cite[Lemma~2.4]{ExMinPar14}. Furthermore, as the operator in question is just the Laplacian with suitable boundary conditions,
the lowest eigenfunction can be represented as
\[
\psi_\alpha(t)=C_1(\alpha)e^{kt}+C_2(\alpha)e^{-kt}, \quad k:=\sqrt{-E_1(T^+_\alpha)}.
\]
This function must satisfy the Dirichlet boundary condition at $t=\delta$,
which gives
\[
C_1(\alpha)=-e^{-2k\delta} C_2(\alpha)=\cO(e^{-2\delta\alpha})C_2(\alpha).
\]
Therefore, for large $\alpha$ we have
\begin{multline*}
1=\|\psi_\alpha\|^2=\int_0^\delta \big|
C_1(\alpha)e^{kt}+C_2(\alpha)e^{-kt}
\big|^2\dd t\\
=
\big|C_1(\alpha)\big|^2\dfrac{e^{2k\delta}-1}{2k}+ \big(\overline{C_1(\alpha)}C_2(\alpha)+C_1(\alpha)\overline{C_2(\alpha)}\big)\delta
+\big|C_2(\alpha)\big|^2\dfrac{1-e^{-2k\delta}}{2k}\\
=\dfrac{1}{2k}\,\big|C_2(\alpha)\big|^2 \big(1+ o(1)\big),
\end{multline*}
from which we infer
\[
\big|C_2(\alpha)\big|^2=2k +o(k)= \cO(\alpha), \quad \big|C_1(\alpha)\big|^2=\cO(k e^{-4\delta\alpha})=\cO(\alpha e^{-4\delta\alpha}),
\]
and, finally,
\[
\big|\psi_\alpha(0)\big|^2= 
\big|C_1(\alpha)\big|^2+ \big(\overline{C_1(\alpha)}C_2(\alpha)+C_1(\alpha)\overline{C_2(\alpha)}\,\big)
+\big|C_2(\alpha)\big|^2=\cO(\alpha).\qedhere
\]
\end{proof}
Now let us pick a point $s_0\in S$ with $m(s_0)=m_\mx$. Take
an arbitrary function $v\in C^\infty\big([0,+\infty)\big)$
which is not identically zero and such that $v'(0)=0$ 
and that $v(x)=0$ for $x\ge 1$.
Consider the functions
\[
u_\alpha(s,t)= v\big( \varepsilon^{-1}d(s,s_0)\big) \psi_\alpha(t),
\]
where $d$ means the geodesic distance on $S$ and $\varepsilon=\varepsilon(\alpha)=o(1)$
as $\alpha\to+\infty$ will be determined later. One easily checks that for $\alpha\to+\infty$ one has
$u_\alpha\in \widetilde H^1_0\big(S\times(0,\delta)\big)$ 
and that
\begin{align}
   \label{eq-th}
\theta(\alpha):=\int_S \Big|v\big( \varepsilon^{-1}d(s,s_0)\big)\Big|^2\dd S\equiv \|u_\alpha\|^2&= c_0 \varepsilon^{\nu-1} +o(\varepsilon^{\nu-1}), \quad c_0>0,\\
\int_{S\times(0,\delta)} g^{\rho\mu} \overline{\partial_\rho u_\alpha}\partial_\mu u_\alpha \dd S \dd t
&=\cO(\varepsilon^{\nu-3}). \nonumber
\end{align}
Furthermore,
\begin{multline*}
\int_{S\times(0,\delta)} |\partial_t u_\alpha|^2 \dd S \dd t -\int_S \big(\alpha+m(s)\big) |u_\alpha(s,0)|^2 \dd S\\
\begin{aligned}
=\,&\int_{S\times(0,\delta)} |\partial_t u_\alpha|^2 \dd S \dd t \\
&{}-(\alpha+m_\mx)\int_S |u_\alpha(s,0)|^2 \dd S
+\int_S \big(m_\mx-m(s)\big) |u_\alpha(s,0)|^2 \dd S\\
=\,&\int_S \Big|v\big( \varepsilon^{-1}d(s,s_0)\big)\Big|^2
\Big(\int_0^\delta \big|\psi_\alpha'(t)\big|^2\dd t
-(\alpha+m_\mx) \big|\psi_\alpha(0)\big|^2\Big)
\dd S\\
&+\big|\psi_\alpha(0)\big|^2
\int_S \big(m_\mx-m(s)\big) \Big|v\big( \varepsilon^{-1}d(s,s_0)\big)\Big|^2 \dd S\\
=\,&E_1(T^+_\alpha)\int_S \Big|v\big( \varepsilon^{-1}d(s,s_0)\big)\Big|^2 \int_0^\delta \big|\psi_\alpha(t)\big|^2\dd t
\dd S\\
&+\big|\psi_\alpha(0)\big|^2
\int_S \big(m_\mx-m(s)\big) \Big|v\big( \varepsilon^{-1}d(s,s_0)\big)\Big|^2 \dd S\\
=\,& - (\alpha+m_\mx)^2 \theta(\alpha) + o\big(\theta(\alpha)\big)\\
&{}+\cO(\alpha) \int_S \big(m_\mx-m(s)\big) \Big|v\big( \varepsilon^{-1}d(s,s_0)\big)\Big|^2 \dd S;
\end{aligned}
\end{multline*}
on the last step we used \eqref{eq-tpl} and \eqref{eq-pp1}.
As $m\in C^1$, one has, for $d(s,s_0)$ sufficiently small,
\begin{equation}
     \label{eq-mm1}
m_\mx-m(s) \le c_1 d(s,s_0),  \quad c_1>0.
\end{equation}
It follows that 
\[
\int_S \big(m_\mx-m(s)\big) \Big|v\big( \varepsilon^{-1}d(s,s_0)\big)\Big|^2 \dd S\le c_1 \varepsilon
\int_S \Big|v\big( \varepsilon^{-1}d(s,s_0)\big)\Big|^2 \dd S=\cO\big(\varepsilon \theta(\alpha)\big).
\]
Finally, using \eqref{eq-th},
\begin{align*}
\dfrac{k^+_\alpha(u_\alpha,u_\alpha)}{\|u_\alpha\|^2}&=
\dfrac{-(\alpha+m_\mx)^2\theta(\alpha) + \cO\big(\theta(\alpha)\big)+\cO(\varepsilon^{\nu-3})+\cO\big(\alpha\varepsilon\theta(\alpha)\big)}{ \theta(\alpha)}\\
&= -(\alpha+m_\mx)^2 + \cO(1)+\cO(\varepsilon^{-2})+\cO(\alpha\varepsilon)\\
&=-\alpha^2-2m_\mx \alpha +\cO(\varepsilon^{-2}+\alpha\varepsilon).
\end{align*}
In order to optimize we take $\varepsilon=\alpha^{-1/3}$, which gives
\begin{equation}
     \label{eq-kplus}
\dfrac{k^+_\alpha(u_\alpha,u_\alpha)}{\|u_\alpha\|^2}\le -\alpha^2-2m_\mx \alpha + \cO\big(\alpha^{2/3}\big).
\end{equation}

Now let us choose $j$ functions $v_1,\dots, v_j \in C^\infty\big([0,+\infty)\big)$
having disjoint supports, non identically zero, and such that $v_i'(0)=0$ 
and that $v_i(x)=0$ for $x\ge 1$ and any $i=1,\dots,j$.
Set
\[
u_{i,\alpha}(s,t)= v_i\big( \varepsilon^{-1}d(s,s_0)\big) \psi_\alpha(t).
\]
One easily checks that
\[
k^+_\alpha(u_{i,\alpha},u_{l,\alpha})=0 \text{ for } i\ne l,
\]
and the preceding computations give the estimate
\[
\dfrac{k^+_\alpha(u_{i,\alpha},u_{i,\alpha})}{\|u_{i,\alpha}\|^2}\le -\alpha^2-2m_\mx \alpha + \cO\big(\alpha^{2/3}\big).
\]
Therefore, by the max-min principle we have
\[
E_j(K^+_\alpha)\le -\alpha^2-2m_\mx \alpha + \cO\big(\alpha^{2/3}\big)
=-\alpha^2 -(\nu-1)H_\mx \,\alpha+ \cO\big(\alpha^{2/3}\big).
\]
By inserting this estimate into \eqref{eq-est3} we get the upper bound of~\eqref{eq-th1}.

If $S\in C^4$, then the remainder estimate can be slightly improved.
Namely, in this case we have $m\in C^2$. As $s_0$ is a local extremum, instead
of \eqref{eq-mm1} one can use 
\begin{equation}
     \label{eq-mm2}
m_\mx-m(s) \le c_1 d(s,s_0)^2,  \quad c_1>0,
\end{equation}
and the choice $\varepsilon:=\alpha^{-1/4}$
transforms \eqref{eq-kplus} into
\[
\dfrac{k^+_\alpha(u_\alpha,u_\alpha)}{\|u_\alpha\|^2}\le -\alpha^2-2m_\mx \alpha + \cO\big(\alpha^{1/2}\big).
\]
and then leads to 
\[
E_j(K^+_\alpha)\le -\alpha^2-2m_\mx \alpha + \cO\big(\alpha^{1/2}\big)
=-\alpha^2 -(\nu-1)H_\mx \,\alpha+ \cO\big(\alpha^{1/2}\big).
\]
This finishes the proof of Theorem~\ref{thm1}.

\section{Estimates for the mean curvature}\label{sec3}

\subsection{Mean curvature and comparison of eigenvalues}

The result of Theorem~\ref{thm1} provides, at least
for large $\alpha$, a geometric point of view at various inequalities involving the Robin eigenvalues.

For example, as shown in \cite{GiSm07}, for all $\alpha>0$ one has
$E_1(Q^\Omega_\alpha)<-\alpha^2$. For large $\alpha$ this has a simple interpretation
in view of Theorem~\ref{thm1}: one always has $H_\mx(\Omega)>0$.

Furthermore, in view of Theorem~\ref{thm1} one has a simple 
condition allowing one to compare
the Robin eigenvalues of two different domains for large $\alpha$: 

\begin{cor}\label{cor5}
Let $\Omega_1$ and $\Omega_2$ be two admissible domains in $\RR^\nu$. 
Assume that
\[
H_\mx(\Omega_1)<H_\mx(\Omega_2),
\]
then for any fixed $j\in\NN$ there exists $\alpha_0>0$ such that for all $\alpha>\alpha_0$ there holds
\[
E(Q^{\Omega_2}_\alpha)<E(Q^{\Omega_1}_\alpha).
\]
\end{cor}

For example, another result obtained in~\cite[Th.~2.4]{GiSm07} is as follows:
if $\Omega$ satisfies a uniform sphere condition of radius $r$
and $B$ is a ball of radius $r/2$, then $E(Q^\Omega_\alpha)\ge E(Q^{B}_\alpha)$ 
for all $\alpha>0$. For large $\alpha$ this can be improved as follows:
\begin{prop}
Let $\Omega\subset \RR^\nu$ be an admissible domain satisfying 
the uniform sphere condition of radius $R>0$
and let $B$ be a ball of radius $r<R$. Then for any $j\in\NN$
one can find $\alpha_0>0$ such that for $\alpha\ge \alpha_0$
there holds $E_j(Q^\Omega_\alpha)>E_j(Q^B_\alpha)$.
\end{prop}

\begin{proof}
The uniform sphere condition of radius $R$ at $s$ implies the inequality
$\max_j \kappa_j(s)\le 1/R$, and due to $H(s)\le \max_j \kappa_j(s)$
we have $ H(s)\le 1/R$ and $H_\mx(\Omega)\le 1/R< 1/r=H_\mx(B)$,
and the assertion follows from Corollary~\ref{cor5}.
\end{proof}

In the following subsections we are going to discuss the problem
of minimizing $H_\mx$ among domains of a fixed volume.
Corollary~\ref{cor5} shows that this geometric problem is closely related
to the reverse Faber-Krahn inequality as discussed in the introduction.
We remark that for a given value of $\vol \Omega$ 
the value $H_\mx(\Omega)$ can be made arbitrarily small
in the class of admissible domains, 
as the example of thin spherical shells shows.
We are going to show that a positive lower bound still exists
in a smaller class in which the ball has the extremal property.

\subsection{Local minimizers}\label{ssec5}

We would like to show first that the balls are the only possible local minimizers of the problem.
Our argument is based on the so-called Alexandrov theorem~\cite{Al62}: Let $\Omega\subset\RR^\nu$
be a bounded, $C^2$ smooth, connected domain such that the mean curvature at the boundary
is constant, then $\Omega$ is a ball.

\begin{prop}
\label{P:localmin}
Let $\Omega\subset\RR^\nu$ be a $C^k$ smooth ($k\ge 2$), compact, connected domain, different from a ball. 
Then for any $\varepsilon>0$ there exists another $C^k$ smooth, compact, connected domain
$\Omega_\varepsilon\subset\RR^\nu$ with the following properties:
\begin{itemize}
\item $\vol\Omega_\varepsilon=\vol\Omega$,
\item the boundary of $\Omega_\varepsilon$ is diffeomorphic to that of $\Omega$,
\item the Hausdorff distance between $\Omega_\varepsilon$ and $\Omega$ is smaller than $\varepsilon$,
\item $H_\mx(\Omega_\varepsilon)<H_\mx(\Omega)$.
\end{itemize}
\end{prop}

\begin{proof}
By Alexandrov's theorem,
the mean curvature $H$ on $\partial \Omega$ is non-constant, so we can find a point
$s_*\in \partial\Omega$ with
$H(s_*)<H_\mx(\Omega)$.
By applying a suitable rotation, without loss of generality we may assume that
for $x=(x',x_\nu)$, $x'=(x_1,\dots,x_{\nu-1})$, close to $s_*$ the condition $x\in\Omega$
is equivalent to $x_\nu< \varphi(x')$, where
$\varphi$ is a $C^k$ function defined in a ball $B\subset\RR^{\nu-1}$,
and hence the boundary $\partial\Omega$ near $s_*$ coincides with the graph of $\varphi$.
By the choice of $s_*$ and using the continuity of $H$ we may additionally
assume that $B$ is chosen in such a way there exists $\delta>0$
such that
\begin{equation}
        \label{eq-hhs}
H(s)\le H_\mx(\Omega)-\delta \text{ for any point }
s\in S_B:=\big\{\big(x',\varphi(x')\big): \,x'\in B\big\}.
\end{equation}

Choose any non-negative $\psi \in C^\infty_c(B)$ which is not identically zero.
If $\varepsilon_0>0$ is sufficiently small, then for all $\varepsilon\in [0,\varepsilon_0)$,
one has the inclusion
\[
\Theta_\varepsilon:=\big\{
(x',x_\nu): \, x'\in B, \, \varphi(x')-\varepsilon \psi(x')<x_\nu<\varphi(x')
\big\}\subset \Omega,
\]
and we set $\widetilde\Omega_\varepsilon:=\Omega\setminus\overline{\Theta_\varepsilon}$.
By construction we have, for $\varepsilon\in(0,\varepsilon_0)$,
\begin{equation}
         \label{eq-vv}
\vol \widetilde\Omega_\varepsilon = \vol \Omega -\varepsilon \int_B\psi(x')dx'  <  \vol \Omega,
\end{equation}
and the Hausdorff distance between $\widetilde\Omega_\varepsilon$ and $\Omega$
is at most $\|\psi\|_{L^\infty(B)}\cdot\varepsilon$.

Furthermore,
\[
\partial\widetilde\Omega_\varepsilon:=(\partial \Omega\setminus S_B)\mathop{\cup}S^\varepsilon_B
\text{ with } S^\varepsilon_B:=\big\{\big(x',\varphi(x')-\varepsilon\psi(x')\big): \,x'\in B\big\},
\]
and $S_B^\varepsilon$ is diffeomorph to $S_B$ (as the both are graphs of smooth functions).
Furthermore, as $\psi$ has a compact support, $S_B^\varepsilon$ coincides with $S_B$ near the boundary,
and, finally, $\partial\widetilde\Omega_\varepsilon$ is diffeomorph to $\partial\Omega$
for any $\varepsilon\in(0,\varepsilon_0)$.

For $x'\in B$, we denote by $H(x',\varepsilon)$ the mean curvature of $\partial\widetilde\Omega_\varepsilon$
at  the point $\big(x',\varphi(x')-\varepsilon\psi(x')\big)$, then
$\big\|H(\cdot,\varepsilon)-H(\cdot,0)\big\|_{L^\infty(B)}=\cO(\varepsilon)$,
and in virtue of~\eqref{eq-hhs} we can find $\varepsilon_1\in(0,\varepsilon_0)$
such that for any $\varepsilon\in(0,\varepsilon_1)$ we have
$H(x',\varepsilon)<H_\mx(\Omega)$ for all $x'\in B$, which gives
$H_\mx(\widetilde \Omega_\varepsilon)=H_\mx(\Omega)$.

Now set
\[
\Omega_\varepsilon:=\Big(\frac{\vol \Omega}{\mathstrut\vol \widetilde\Omega_\varepsilon}\Big)^{1/\nu}\widetilde\Omega_\varepsilon,
\]
then $\vol\Omega_\varepsilon=\vol \Omega$, and
\[
H_{\mx}(\Omega_\varepsilon)=
\Big(\frac{\vol \widetilde \Omega_\varepsilon}{\vol \Omega}\Big)^{1/\nu} H_{\mx}(\widetilde\Omega_\varepsilon)
\equiv
\Big(\frac{\vol \widetilde \Omega_\varepsilon}{\vol \Omega}\Big)^{1/\nu} H_{\mx}(\Omega)
< H_{\mx}(\Omega).
\]
The boundary of $\Omega_\varepsilon$ is clearly diffeomorph
to that of $\widetilde\Omega_\varepsilon$ and, consequently,
to that of $\Omega$, and using~\eqref{eq-vv}
one easily checks that that the Hausdorff distance
between $\Omega_\varepsilon$ and $\Omega$ is $\cO(\varepsilon)$.
\end{proof}

\subsection{Proof of Theorem \ref{thm2}}\label{ssec6}
We now prove Theorem \ref{thm2}. Without loss of generality we assume that $\Omega$ is star-shaped with respect to the origin.
Let us introduce a function $p$ on $S$ by $p(s):=s\cdot n(s)$, where $\cdot$ stands for the scalar product in $\RR^\nu$.
The function $p$ is called sometimes the support function.
We will use two integral identities related to $\Omega$.
First, due to the divergence theorem we have
\[
\vol \Omega=\frac{1}{\nu}\int_{S}p \, \dd S .
\]
Furthermore, we have the so-called Minkowski formula
\[
\area S=\int_{S} p\,H\,\dd S,
\]
see \cite[Theorem 1]{Hs54}; remark that in \cite{Hs54}, the mean curvature $M_{1}$ corresponds to $(-H)$ with our notation. 

As $\Omega$ is star-shaped, we have $p \ge 0$, and we deduce 
\[
\area S \le H_\mx(\Omega) \int_S p\, \dd S = \nu H_\mx(\Omega)  \vol \Omega,
\]
which gives
\begin{equation}
        \label{eq-bb1}
H_{\mx}(\Omega) \ge \frac{1}{\nu}\frac{\area S}{\vol \Omega}.
\end{equation}
Due to the classical isoperimetric inequality we have 
\begin{equation}
        \label{eq-bb2}
\area S\geq \nu \Big(\dfrac{\vol B_\nu}{\vol \Omega}\Big)^{1/\nu} \vol \Omega,
\end{equation}
and the equality holds only if $\Omega$ is a ball~\cite[\S 2.10]{Bur}.
Substituting \eqref{eq-bb2} into \eqref{eq-bb1} we obtain
\begin{equation}
      \label{eq-bb3}
H_\mx(\Omega)\ge \left(\frac{\vol B_\nu}{\vol \Omega}\right)^{1/\nu}.
\end{equation}
If one has the equality in \eqref{eq-bb3}, then automatically one has the equality in~\eqref{eq-bb2},
which gives in turn that $\Omega$ is a ball. On the other hand, is $\Omega$ is a ball
of radius $r>0$, then $H_\mx(\Omega)=1/r$ and $\vol\Omega= r^\nu \vol B_\nu$, and 
one has the equality in~\eqref{eq-bb3}.

\section*{Acknowledgments}
The research was partially supported by ANR NOSEVOL and GDR DynQua.
We thank Lorenzo Brasco, Pierre Pansu and Nader Yeganefar for motivating
discussions.

\end{document}